\documentclass[11pt]{elsarticle}
\usepackage{}

\usepackage{amssymb}
\usepackage{xcolor}
\usepackage{mathrsfs}
\usepackage{amscd}
\usepackage{appendix}
\usepackage{geometry}
\usepackage{setspace}
\usepackage{amsfonts}
\usepackage{amsmath}
\usepackage{lineno,hyperref}
\usepackage{tikz}
\usetikzlibrary{matrix,arrows}

\usepackage{amssymb,amscd,amsthm,verbatim}
\biboptions{sort&compress}

\usepackage{bbm}

\date{\today}

\renewcommand{\Re}{{\mathrm{Re} \,}}
\renewcommand{\Im}{{\mathrm{Im} \,}}

\newtheorem{theorem}{Theorem}[section]
\newtheorem{lemma}[theorem]{Lemma}

\theoremstyle{definition}

\newtheorem*{definition 1}{Definition 1}
\newtheorem*{definition 2}{Definition 2}
\newtheorem*{definition 3}{Definition 3}
\newtheorem*{definition 4}{Definition 4}
\newtheorem*{definition 5}{Definition 5}

\newtheorem{remark}[theorem]{Remark}

\theoremstyle{plain}
\newtheorem{definition}{Definition}[section]
\allowdisplaybreaks \numberwithin{equation}{section}

\journal{arXiv}

\begin{document}
\title{Anderson localization for the multi-frequency quasi-periodic \\ CMV matrices and quantum walks}

\author{Bei Zhang}
\ead{beizhang@nankai.edu.cn}
\address{Chern Institute of Mathematics and LPMC, Nankai University, Tianjin 300071, China}

\author{Daxiong Piao\corref{cor1}}
\ead{dxpiao@ouc.edu.cn}
\address{School of Mathematical Sciences,  Ocean University of China, Qingdao 266100, China}

\cortext[cor1]{Corresponding author}

\begin{abstract}
In this paper we prove Anderson localization for multi-frequency quasi-periodic extended CMV matrices with analytic Verblunsky coefficients in the regime of positive Lyapunov exponents.  By constructing a suitable semialgebraic set and combining the Avalanche Principle with a Large Deviation Theorem, we overcome the key obstruction of eliminating double resonances along the orbit, where multi-frequency potentials introduce significant challenges compared to the single-frequency case.  As a direct application, we establish Anderson localization for corresponding analytic multi-frequency quasi-periodic quantum walks via unitary equivalence.

\end{abstract}

\maketitle


\section{Introduction and Preliminaries}\label{sec:Intro}
\subsection{Research Background}
The study of Anderson localization (AL) for CMV matrices and quantum walks has garnered significant recent interest. Zhu \cite{Zhu-JAT}\footnote{In the paper \cite{Zhu-JAT}, Zhu also corrects some of the existing CMV related issues in the literature. Our present paper fully considers these corrections.} established both AL and strong dynamical localization for the random CMV matrices. In the quasi-periodic setting, Wang and Damanik \cite{WD19-JFA} formulated AL for CMV matrices and quantum walks with Verblunsky coefficients defined by a simple shift. Cedzich and Werner \cite {Cedzich-2021} obtained AL for a class of electric quantum walks and skew-shift CMV matrices. Subsequently, Lin et al. \cite{LPG23-JFA} proved AL for CMV matrices with Verblunsky coefficients defined by a skew-shift, extending localization results for Schr\"{o}dinger operators from \cite{BGS01-CMP}. Following this, Lin et al. \cite{LPG24-arXiv}  also demonstrated AL for CMV matrices where the Verblunsky coefficients are generated by a hyperbolic toral automorphism.

To the best of our knowledge, the localization problem for multi-frequency quasi-periodic CMV matrices remains open. Addressing this gap is the primary focus of the present paper. Our motivation stems strongly from the extensive body of work on AL and related spectral analysis for multi-frequency quasi-periodic Schr\"{o}dinger operators \cite{CS-CMP-1989, BG00-Annals, Bourgain-JAM, HWZ, Zhao20-JFA, GSV16-arXiv, GSV19-Inventiones}. A fundamental observation in the spectral theory of quasi-periodic operators is that the analysis of AL for multi-frequency potentials presents considerably greater complexity than the one-frequency case. A principal obstruction lies in eliminating double resonances along the orbit, a well-documented challenge \cite{Bourgain-book, CS-CMP-1989}. Our investigation reveals that the situation for CMV matrices is analogous. While close connections exist between the spectral theories of CMV operators and Schr\"{o}dinger operators, translating results from the Schr\"{o}dinger setting to the CMV case is often non-trivial.

This non-triviality is particularly evident in the multi-frequency context. For the one-frequency quasi-periodic CMV matrices, Wang and Damanik \cite{WD19-JFA} were able to directly eliminate double resonances using a frequency estimate \cite[Lemma 6.1]{BG00-Annals}. However, no analogous estimate is currently available for the multi-frequency CMV setting. To overcome this hurdle, we employ powerful tools from the theory of semialgebraic sets, specifically \cite[Corollary 9.7 and Lemma 9.9]{Bourgain-book}. Consequently, a key step in our approach is the construction of a suitable semialgebraic set and its decomposition. This strategy, combined with the application of the Avalanche Principle (AP) and a Large Deviation Theorem (LDT), forms the core of our proof.

For completeness, we note that significant progress has also been made on other spectral properties of CMV matrices, such as the regularity of the Lyapunov exponent and the nature of the absolutely continuous spectrum; see, for instance, \cite{Damanik-Kruger-2009,Damanik-Lenz-2007,Fang-JFA, Li-Damanik-Zhou,YangF-2024, Zhang-Nonlinearity}.

The rest of the present paper is organized as follows: Section \ref{sec:Intro} provides a quick introduction to the orthogonal polynomials on the unit circle (OPUC) and the settings of CMV matrices, Szeg\"{o} cocycle, and Lyapunov exponents. Section \ref{sec:tool} contains some basic tools (including the Avalanche Principle, Large Deviation Theorem, estimates for the Lyapunov exponent, Green's function, Poisson's formula, and semialgebraic sets) which are helpful to the following study. Section \ref{sec:proof} addresses the proof of the main theorem (Theorem \ref{main-theorem}). Finally, Section \ref{sec:appl} applies these results to derive AL for the analytic multi-frequency quasi-periodic quantum walks, utilizing the Gesztesy-Zinchenko cocycle framework.

 \subsection{CMV Matrices and Key Concepts}

Let us recall some preliminaries of OPUC that can be seen for more details in  Simon's monograph \cite{Simon-book}.

Let $\mathbb{D}=\{z: |z|<1\}$ be the open unit disk in $\mathbb{C}$ and $\mu$ be a nontrivial probability measure on $\partial \mathbb{D}=\{z: |z|=1\}$.  Then the functions $1, z, z^{2}, \ldots$ are linearly independent in the Hilbert space $L^{2}(\partial \mathbb{D}, d\mu)$. Let $\Phi_{n}(z)$ be the monic orthogonal polynomials, that is,
$\Phi_{n}(z)=P_{n}[z^{n}]$,  $P_{n}\equiv$ projection onto $\{1, z, z^{2}, \ldots, z^{n-1}\}^{\bot}$ in $L^{2}(\partial \mathbb{D}, d\mu)$. Then the orthonormal polynomials are
$$\varphi_{n}(z)=\frac{\Phi_{n}(z)}{\|\Phi_{n}(z)\|_{\mu}},$$
where $\|\cdot\|_{\mu}$ denotes the norm of $L^{2}(\partial \mathbb{D}, d\mu)$.

For any polynomial $Q_{n}(z)$ of degree $n$, its Szeg\H{o} dual $Q_{n}^{*}(z)$ is defined by
$$Q_{n}^{*}(z)=z^{n}\overline{Q_{n}(1/\overline{z})}.$$
Then Szeg\H{o} recursion is
\begin{equation}\label{Szego-1}
\Phi_{n+1}(z)=z\Phi_{n}(z)-\overline{\alpha}_{n}\Phi_{n}^{*}(z),
\end{equation}
where the parameters $\{\alpha_{n}\}_{n=0}^{\infty}$ in $\mathbb{D}$ are called the Verblunsky coefficients.

Orthonormalizing $1, z, z^{-1}, z^{2}, z^{-2},\cdots$, yields a CMV basis $\{\chi_{j}\}_{j=0}^{\infty}$ of $L^{2}(\partial \mathbb{D},d\mu)$, and the matrix representation of multiplication by $z$ relative to the CMV basis gives rise to the CMV matrix $\mathcal{C}$,
$$\mathcal{C}_{ij}=\langle \chi_{i},z\chi_{j} \rangle.$$
Then $\mathcal{C}$ has the form
\begin{equation*}
\mathcal{C}=
\left(
\begin{array}{ccccccc}
\overline{\alpha}_{0}&\overline{\alpha}_{1}\rho_{0}&\rho_{1}\rho_{0}&  &  & & \\
\rho_{0}&-\overline{\alpha}_{1}\alpha_{0}&-\rho_{1}\alpha_{0}& & & &\\
& \overline{\alpha}_{2}\rho_{1} &-\overline{\alpha}_{2}\alpha_{1}& \overline{\alpha}_{3}\rho_{2}&\rho_{3}\rho_{2}& & \\
& \rho_{2}\rho_{1}& -\rho_{2}\alpha_{1}&-\overline{\alpha}_{3}\alpha_{2}&-\rho_{3}\alpha_{2}& & \\
& & &\overline{\alpha}_{4}\rho_{3}&-\overline{\alpha}_{4}\alpha_{3}&\overline{\alpha}_{5}\rho_{4}& \\
& & &\rho_{4}\rho_{3}&-\rho_{4}\alpha_{3}&-\overline{\alpha}_{5}\alpha_{4}& \\
& & & & \ddots & \ddots & \ddots
\end{array}
\right),
\end{equation*}
where $\rho_{j}=\sqrt{1-|\alpha_{j}|^{2}}$, $j\geq 0$. Similarly, we can get the extended CMV matrix
\begin{equation*}
\mathcal{E}=
\left(
\begin{array}{cccccccc}
\ddots&\ddots&\ddots& & & & & \\
& -\overline{\alpha}_{0}\alpha_{-1}&\overline{\alpha}_{1}\rho_{0}&\rho_{1}\rho_{0}&  &  & & \\
& -\rho_{0}\alpha_{-1}&-\overline{\alpha}_{1}\alpha_{0}&-\rho_{1}\alpha_{0}& & & &\\
& & \overline{\alpha}_{2}\rho_{1} &-\overline{\alpha}_{2}\alpha_{1}& \overline{\alpha}_{3}\rho_{2}&\rho_{3}\rho_{2}& & \\
& & \rho_{2}\rho_{1}& -\rho_{2}\alpha_{1}&-\overline{\alpha}_{3}\alpha_{2}&-\rho_{3}\alpha_{2}& & \\
& & & &\overline{\alpha}_{4}\rho_{3}&-\overline{\alpha}_{4}\alpha_{3}&\overline{\alpha}_{5}\rho_{4}& \\
& & & &\rho_{4}\rho_{3}&-\rho_{4}\alpha_{3}&-\overline{\alpha}_{5}\alpha_{4}& \\
& & & & & \ddots & \ddots & \ddots
\end{array}
\right),
\end{equation*}
which is a special five-diagonal doubly-infinite unitary matrix in the standard
basis of $L^{2}(\partial \mathbb{D}, d\mu)$ according to \cite[Subsection 4.5]{Simon-book} and \cite[Subsection 10.5]{Simon-book2}.

\subsection{Szeg\H{o} Cocycle and Lyapunov Exponent}\label{secSCLE}

In this paper, we consider the extended CMV matrix with Verblunsky coefficients generated by an analytic function $\alpha(\cdot):\mathbb{T}^{d}\rightarrow \mathbb{D}$, $d>1$, $\alpha_{n}(x)=\alpha(T^{n}x)=\alpha(x+n\omega)$, where $T: \mathbb{T}^{d}\rightarrow \mathbb{T}^{d} $ is the \emph{shift} map of the form $Tx=x+\omega$, $\mathbb{T}:=\mathbb{R}/\mathbb{Z}$, $x, \omega\in \mathbb{T}^{d}$ are  phase and frequency, respectively. Moreover, the frequency vector $\omega\in \mathbb{T}^{d}$ satisfies the standard Diophantine condition
\begin{equation}\label{standard-DC}
\|k\cdot \omega\|\geq \frac{p}{|k|^{q}}
\end{equation}
for all nonzero $k\in \mathbb{Z}^{d}$, where $p>0$, $q>d$ are some constants, $\|\cdot\|$ denotes the distance to the nearest integer and $|\cdot|$ stands for the sup-norm on $\mathbb{Z}^{d}$ ($|k|=|k_{1}|+|k_{2}|+\cdots+|k_{d}|$, $k_{i}$ is the $i$-th element of the vector $k$).

For the sake of convenience, the set of $\omega$ satisfying the standard Diophantine condition  \eqref{standard-DC} is denoted by $\mathbb{T}^{d}(p,q)\subset \mathbb{T}^{d}$.

Assume that the sampling function $\alpha(x)$ satisfies that
\begin{equation}\label{sampling-function}
\int_{\mathbb{T}^{d}}\log(1-|\alpha(x)|)dx>-\infty
\end{equation}
and it can extend complex analytically to the region
\begin{equation*}
\mathbb{T}^{d}_{h}:=\{x+iy:x\in \mathbb{T}^{d}, y\in \mathbb{R}^{d}, |y|<h\}
\end{equation*}
for some constant $h>0$.

Then the Szeg\H{o} recursion is equivalent to
\begin{equation}\label{Szego-2}
\rho_{n}(x)\varphi_{n+1}(z)=z\varphi_{n}(z)-\overline{\alpha}_{n}(x)\varphi_{n}^{*}(z),
\end{equation}
where $\rho_{n}(x)=\rho(x+n\omega)$ and $\rho(x)=(1-|\alpha(x)|^{2})^{1/2}$. Applying Szeg\H{o} dual to both sides of (\ref{Szego-2}), one can obtain that
\begin{equation}\label{Szego-3}
\rho_{n}(x)\varphi_{n+1}^{*}(z)=\varphi_{n}^{*}(z)-\alpha_{n}(x)z\varphi_{n}(z).
\end{equation}

Then one can get the matrix form of  (\ref{Szego-2}) and (\ref{Szego-3}), i.e.,
\begin{equation*}
\left(
\begin{array}{c}
\varphi_{n+1}\\
\varphi_{n+1}^{*}
\end{array}
\right)
=
S(\omega,z;x+n\omega)
\left(
\begin{array}{c}
\varphi_{n}\\
\varphi_{n}^{*}
\end{array}
\right),
\end{equation*}
where
\begin{equation*}
S(\omega,z;x)=\frac{1}{\rho(x)}
\left(
\begin{array}{cc}
z&-\overline{\alpha}(x)\\
-\alpha(x)z&1
\end{array}
\right).
\end{equation*}
Since $\det S(\omega,z;x)=z$, one always prefers to study the  matrix
\begin{equation*}
M(\omega,z;x)=\frac{1}{\rho(x)}
\left(
\begin{array}{cc}
\sqrt{z}&-\frac{\overline{\alpha}(x)}{\sqrt{z}}\\
-\alpha(x)\sqrt{z}&\frac{1}{\sqrt{z}}
\end{array}
\right)\in \mathbb{S}\mathbb{U}(1,1).
\end{equation*}
This is called the Szeg\H{o} cocycle map. It is easy to verify that $\mathrm{det}M(\omega,z;x)=1$. Then the monodromy matrix (or $n$-step transfer matrix) is defined by
\begin{equation}\label{n-step}
M_{n}(\omega,z;x)=\prod_{j=n-1}^{0}M(\omega,z;x+j\omega).
\end{equation}
According to (\ref{n-step}), it is obvious that
\begin{equation*}
M_{n_{1}+n_{2}}(\omega,z;x)=M_{n_{2}}(\omega,z;x+n_{1}\omega)M_{n_{1}}(\omega,z;x)
\end{equation*}
and
\begin{equation}\label{log-subadditive}
\log\|M_{n_{1}+n_{2}}(\omega,z;x)\|\leq \log\|M_{n_{1}}(\omega,z;x)\|+\log\|M_{n_{2}}(\omega,z;x+n_{1}\omega)\|.
\end{equation}
Let
\begin{equation}\label{un}
u_{n}(\omega,z;x):=\frac{1}{n}\log\|M_{n}(\omega,z;x)\|
\end{equation}
and
\begin{equation}\label{un}
L_{n}(\omega,z):=\int_{\mathbb{T}^{d}}u_{n}(\omega,z;x)dx.
\end{equation}
Integrating the inequality (\ref{log-subadditive}) with respect to $x$ over $\mathbb{T}^{d}$, we have that
\begin{equation*}
L_{n_{1}+n_{2}}(\omega,z)\leq \frac{n_{1}}{n_{1}+n_{2}}L_{n_{1}}(\omega,z)+\frac{n_{2}}{n_{1}+n_{2}}L_{n_{2}}(\omega,z).
\end{equation*}
This implies that
\begin{equation*}
L_{n}(\omega,z)\leq L_{m}(\omega,z) \quad \mathrm{if} \quad m<n,\; m|n
\end{equation*}
and
\begin{equation*}
L_{n}(\omega,z)\leq L_{m}(\omega,z)+C\frac{m}{n} \quad \mathrm{if} \quad m<n.
\end{equation*}

As we know, the transformation $x\rightarrow x+\omega$ is ergodic for any irrational $\omega\in \mathbb{T}^{d}$. Notice that (\ref{log-subadditive}) implies that $\log \|M_{n}(\omega,z;x)\|$ is subadditive. Thus, according to Kingman's subadditive ergodic theorem, the limit
\begin{equation}\label{Lyapunov-exponent}
L(\omega,z)=\lim_{n\rightarrow\infty}L_{n}(\omega,z)
\end{equation}
exists. This is called the Lyapunov exponent. Throughout this paper, we let $\gamma$ be the lower bound of the Lyapunov exponent. On the other hand, Furstenberg-Kesten theorem indicates that the limit also exists for $ a.e.\;x$:
\begin{equation}\label{Lyapunov-exponent-1}
\lim_{n\rightarrow\infty}u_{n}(\omega,z;x)=\lim_{n\rightarrow\infty}L_{n}(\omega,z)=L(\omega,z).
\end{equation}
\begin{definition}[Generalized eigenvalues and generalized eigenfunctions] We call $z$ a generalized eigenvalue of the extended CMV matrix $\mathcal{E}$ , if there exists a nonzero, polynomially bounded function $u$ such
that $\mathcal{E}= z u$ where $u = \{u(n)\}_{n\in \mathbb{Z}}$. We call $u$ a generalized eigenfunction.
\end{definition}

\begin{definition}[Anderson localization] We say that extended CMV matrix $\mathcal{E}$ exhibits Anderson localization on
$\mathcal{I} \subset \partial\mathbb{D}$ , if $\mathcal{E}$ has only pure point spectrum in $\mathcal{I}$ and for every eigenvalue its eigenfunction $u$ decay exponentially in $n$.
\end{definition}

 \subsection{Main Results}

Our main results are following theorem and applications to quantum walks.
\begin{theorem}\label{main-theorem}
Let $\mathcal{I_{\alpha}} \subset \partial\mathbb{D}$ be a compact spectral interval. Consider the extended CMV matrices $\mathcal{E}(x)$ with Verblunsky coefficients $\alpha_n(x) = \alpha(x + n\omega)$, where $\alpha: \mathbb{T}^d \to \mathbb{D}$ is analytic. Assume that for every $z \in \mathcal{I_{\alpha}}$ and every $\omega \in \mathbb{T}^{d}(p,q)$, there exists a constant $\gamma > 0$ such that the Lyapunov exponent satisfies $L(\omega, z) > \gamma > 0$. Then for fixed $x_0$ and a.e. $\omega \in \mathbb{T}^{d}(p,q)$, $\mathcal{E}(x_0)$ exhibits Anderson localization on $\mathcal{I_{\alpha}}$.
\end{theorem}

This theorem has applications to a class of quantum walks; see Theorem \ref{thm:qw-al}.

\section{Basic Tools}\label{sec:tool}
In this section, we require some useful lemmas which could have applications in the following research. To begin with, we introduce some useful notations\cite[Definition 2.3]{BGS02-Acta}.

For any positive numbers $a, b$ the notation $a\lesssim b$ means $Ca\leq b$ for some constant $C>0$. By $a\ll b$ we mean that the constant $C$ is very large. If both $a\lesssim b$ and $a\gtrsim b$, then we write $a\asymp b$.
\subsection{Avalanche Principle}
\begin{lemma}\label{avalanche-principle}\cite[Proposition 2.2]{GS01-Annals}
Let $A_{1},\ldots,A_{m}$ be a sequence of arbitrary unimodular $2\times 2$-matrices. Suppose that
\begin{equation}\label{AP-1}
\min_{1\leq j\leq m}\|A_{j}\|\geq \mu\geq m
\end{equation}
and
\begin{equation}\label{AP-2}
\max_{1\leq j< m}\big[\log\|A_{j+1}\|+\log\|A_{j}\|-\log\|A_{j+1}A_{j}\|\big]<\frac{1}{2} \log\mu.
\end{equation}
Then
\begin{equation}\label{AP-3}
\Big|\log\|A_{m}\cdots A_{1}\|+\sum_{j=2}^{m-1}\log\|A_{j}\|-\sum_{j=1}^{m-1}\log\|A_{j+1}A_{j}\|\Big|<C_{A}\frac{m}{\mu},
\end{equation}
where $C_{A}$ is an absolute constant.
\end{lemma}

\subsection{Large Deviation Theorem}
Now we turn attention to proving the large deviation theorem which is the most basic tool in the theory of localization for CMV matrices. To obtain LDT, we need the following lemma.
\begin{lemma}\label{Lemma3.1}\cite[Proposition 9.1]{GS01-Annals}
Let $d$ be a positive integer. Suppose $u: D(0,2)^{d}\rightarrow [-1,1]$ is subharmonic in each variable; i.e., $z_{1}\rightarrow u(z_{1},z_{2},\ldots,z_{d})$ is subharmonic for any choice of $(z_{2},\ldots,z_{d})\in D(0,2)^{d-1}$ and similarly for each of the other variables. Assume furthermore that for some $n\geq 1$
\begin{equation}\label{GS01-(1)}
\sup_{\theta \in \mathbb{T}^{d}}|u(\theta+\omega)-u(\theta)|<\frac{1}{n}.
\end{equation}
Then there exist $\sigma>0$, $\tau>0$, and $c_{0}$ only depending on $d$ and $\varepsilon_{1}$ such that
\begin{equation}\label{GS01-(2)}
\mathrm{mes}\{\theta\in \mathbb{T}^{d}: |u(\theta)-\langle u\rangle|> n^{-\tau}\}<\exp(-c_{0}n^{\sigma}).
\end{equation}
Here $\langle u\rangle=\int_{\mathbb{T}^{d}}u(\theta)d\theta$. If $d=2$ then the range is $0<\tau<\frac{1}{3}-\varepsilon_{2}$ and $\sigma=\frac{1}{3}-\tau-\varepsilon_{2}$ where $\varepsilon_{2}\rightarrow 0$ as $\varepsilon_{1}\rightarrow 0$.
\end{lemma}
Based on the above lemma, we get the large deviation estimate.
\begin{theorem}\label{(LDT)-matrix}
Let $\mathcal{I_{\alpha}} \subset \partial\mathbb{D}$ be a compact spectral interval. Suppose that for all $\omega\in \mathbb{T}^{d}(p,q)$ and $z\in \mathcal{I_{\alpha}} \subset \mathbb{D}$, $L(\omega,z)>\gamma>0$. There exist $\sigma=\sigma(p,q)$, $\tau=\tau(p,q)$, $\sigma, \tau \in (0,1)$, $C_{0}=C_{0}(p,q,h)$ such that for $n\geq 1$ one has that
\begin{equation}\label{(LDT)1}
\mathrm{mes}\{x\in \mathbb{T}^{d}: |\log\|M_{n}(\omega,z;x)\|-nL_{n}(\omega,z)|> n^{1-\tau}\}<\exp(-C_{0}n^{\sigma}).
\end{equation}
\end{theorem}
\begin{proof}
Fix some dimension $d$ and $z$. For any $x\in D(0,2)^{d}$, define
$$u_{n}(x)=\frac{1}{n}\log\|M_{n}(\omega,z;x)\|.$$
Then $u_{n}$ is a continuous subharmonic function which is bounded by $1$ in $D(0,1)^{d}$. In addition, $u_{n}$ satisfies the conditions in the above lemma. Thus, \eqref{(LDT)1} is an immediate consequence of \eqref{GS01-(2)}.
\end{proof}

\subsection{Estimate for the Lyapunov Exponent}
\begin{lemma}\label{Lemma3.4}
Assume that for all $\omega\in \mathbb{T}^{d}(p,q)$ and $z\in \mathcal{I_{\alpha}} \subset \mathbb{D}$, $L(\omega,z)>\gamma>0$. Then for any $n\geq 2$,
\begin{equation*}
0\leq L_{n}(\omega,z)-L(\omega,z)<C\frac{(\log n)^{1/\sigma}}{n},
\end{equation*}
where $C=C(p,q,z,\gamma)$ and $\sigma$ is as in LDT.
\end{lemma}
\begin{proof}
Obviously, $0\leq L_{n}(\omega,z)\leq C(z)$ for all $n$. Let $k$ be a positive integer such that $k\gamma>16C(z)$. Given $n_{0}>10$, let $l_{0}=[C_{1}(\log n_{0})^{1/\sigma}]$ with some large $C_{1}$. Consider the integers $l_{0}, 2l_{0}, \ldots,2^{k}l_{0}$. Then there exists some $0\leq j<k$ such that (with $l_{j}=2^{j}l_{0}$)
\begin{equation}\label{lem3.4-(1)}
L_{l_{j}}(z)-L_{l_{j+1}}(z)<\frac{\gamma}{16}.
\end{equation}
If not, then $C(z)>L_{l_{0}}(z)-L_{2^{k}l_{0}}(z)\geq \frac{k\gamma}{16}>C(z)$, which is a contradiction. Let $l=2^{j}l_{0}$ with the choice of $j$ satisfying \eqref{lem3.4-(1)}.

It is not difficult to check that $M(\omega,z;x)$ is conjugate to an $\mathbb{SL}(2,\mathbb{R})$ matrix $T(\omega,z;x)$ through the matrix
\begin{equation*}
Q=-\frac{1}{1+i}
\left(
\begin{array}{cc}
1&-i\\
1&i
\end{array}
\right)\in \mathbb{U}(2);
\end{equation*}
that is,
\begin{equation*}
T(\omega,z;x)=Q^{*}M(\omega,z;x)Q\in \mathbb{SL}(2,\mathbb{R}),
\end{equation*}
where ``$*$'' stands for the conjugate transpose of a matrix.

Now we apply AP to the matrices $B_{j}=T_{l}(\omega,z;x+(j-1)l\omega)$ for $1\leq j\leq m=[n/l]$ with $\mu=\exp(l\gamma/2)$. Note that $\mu>n^{2}$ if $C_{1}>4/\gamma$. According to LDT, the phases $x$ such that
\begin{equation*}
\min_{1\leq j\leq m} \|B_{j}\|\geq \exp(lL_{l}(\omega,z)-l^{1-\tau})>\exp(l\gamma/2)=\mu>n^{2}
\end{equation*}
form a set $\mathcal{G}_{1}$ whose measure is not exceeding
$$m \exp(-C_{0}l^{\sigma})\leq n\exp(-C_{0}(2^{j}l_{0})^{\sigma})<n^{-2}$$
provided $C_{1}>n^{1/\sigma}$.

Furthermore, combining \eqref{lem3.4-(1)} and LDT,
\begin{align*}
\max_{1\leq j< m}\Big|\log\|B_{j+1}\|+\log\|B_{j}\|-\log\|B_{j+1}B_{j}\| \Big|
&\leq 2l(L_{l}(z)+\frac{\gamma}{32})-2l(L_{2l}(z)-\frac{\gamma}{32})\\
&=2l[(L_{l}(z)-L_{2l}(z))+\frac{\gamma}{16})]\\
&\leq \exp(l\gamma/4)=\frac{1}{2}\log \mu
\end{align*}
up to a set $\mathcal{G}_{2}$ of $x$ of measure not exceeding
$$2m \exp(-C_{0}l^{\sigma})\leq 2n\exp(-C_{0}(l)^{\sigma})<n^{-2}$$
if $C_{1}>n^{1/\sigma}$. Let $\mathcal{G}=\mathcal{G}_{1}\cup \mathcal{G}_{2}$. Then $\mathrm{mes}\mathcal{G}<2n^{-2}$ and \eqref{AP-1}, \eqref{AP-2} hold for all $x \in\mathcal{G}$.

Based on AP, we have that
\begin{equation*}
\Big|\log\|B_{m}\cdots B_{1}\|+\sum_{j=2}^{m-1}\log\|B_{j}\|-\sum_{j=1}^{m-1}\log\|B_{j+1}B_{j}\|\Big|<C\frac{n}{\mu}<Cn^{-1}
\end{equation*}
for all $x\in\mathcal{G}$. Equivalently, one can obtain that
\begin{small}
\begin{equation*}
\Big|\log\|T_{lm}(\omega,z;x)+\sum_{j=2}^{m-1}\log\|T_{l}(\omega,z;x+jl\omega)\|-\sum_{j=1}^{m-1}\log\| T_{l}(\omega,z;x+(j+1)l\omega)T_{l}(\omega,z;x+jl\omega) \|\Big|<Cn^{-1}
\end{equation*}
\end{small}
for all $x\in\mathcal{G}$.

Similarly, for $\log \|T_{lm}(\omega,z;x+lm\omega)\|$ and $\log \|T_{lm}(\omega,z;x)\|$, the phases $x$ such that
\begin{align}	
&\Big| \|\log\| T_{2lm}(\omega,z;x)\|-\|\log\| T_{lm}(\omega,z;x+lm\omega)\|-\|\log\| T_{lm}(\omega,z;x)\| \nonumber\\
&+\|\log\| T_{l}(\omega,z;x+lm\omega)\|+\|\log \|T_{l}(\omega,z;x+(m-1)l\omega)\| \nonumber \\
&-\|\log \|T_{l}(\omega,z;x+lm\omega)T_{l}(\omega,z;x+(m-1) l\omega)\|\Big|\leq \frac{C}{n} \label{Lem3.4-(2)}
\end{align}
form a set with measure not exceeding $Cn^{-2}$.\\
Since
\begin{equation}\label{Lem3.4-(3)}
\big|\|\log\| T_{n}(\omega,z;x)\|-\log\| T_{lm}(\omega,z;x)\| \big|\leq C(\alpha_{n},\alpha_{lm},z)n
\end{equation}
and
\begin{equation}\label{Lem3.4-(4)}
\big|\|\log\| T_{l}(\omega,z;x)\| \big|\leq C(z)l,
\end{equation}
we can conclude from \eqref{Lem3.4-(2)} that
\begin{equation*}
\big|\|\log\| T_{2n}(\omega,z;x)\|-\log\| T_{n}(\omega,z;x+n\omega)\|-\log\| T_{n}(\omega,z;x)\| \big|\leq C(\log n)^{1/\sigma}
\end{equation*}
up to a set of $x$ not exceeding $Cn^{-2}$ in measure.

Integrating the above inequality over $x$, we have that
\begin{equation*}
|L_{2n}(z)-L_{n}(z)|\leq C\frac{(\log n)^{1/\sigma}}{n},
\end{equation*}
where $C=C(\gamma,z,\omega)$.

We can finally prove this lemma by summing over $2^{k}n$ .
\end{proof}
\subsection{Green's Function and Poisson's Formula}\label{3.4}
Define the unitary matrices
\begin{equation*}
\Theta_{n}=
\left(
\begin{array}{cc}
\overline{\alpha}_{n}&\rho_{n}\\
\rho_{n}&-\alpha_{n}
\end{array}
\right).
\end{equation*}
Then one can factorize the matrix $\mathcal{C}$ as
\begin{equation*}
\mathcal{C}=\mathcal{L}_{+}\mathcal{M}_{+},
\end{equation*}
where
\begin{equation*}
\mathcal{L}_{+}=\left(
\begin{matrix}
\Theta_0 &~ & ~\\
~& \Theta_2 & ~\\
~ & ~& \ddots
\end{matrix}
\right),\quad
\mathcal{M}_{+}=\left(
\begin{matrix}
\mathbf{1} &~ & ~\\
~& \Theta_1 & ~\\
~ & ~& \ddots
\end{matrix}
\right).
\end{equation*}

Similarly, the extended CMV matrix can be written as
\begin{equation*}
\mathcal{E}=\mathcal{L}\mathcal{M},
\end{equation*}
where
\begin{equation*}
\mathcal{L}=\bigoplus_{j\in \mathbb{Z}}\Theta_{2j}, \quad \mathcal{M}=\bigoplus_{j\in \mathbb{Z}}\Theta_{2j+1}.
\end{equation*}

We let $\mathcal{E}_{[a,b]}$ denote the restriction of an extended CMV matrix to a finite interval $[a,b]\subset\mathbb{Z}$, defined by
\begin{equation*}
\mathcal{E}_{[a,b]}=P_{[a,b]}\mathcal{E}(P_{[a,b]})^{*},
\end{equation*}
where $P_{[a,b]}$ is the projection $\ell^{2}(\mathbb{Z})\rightarrow \ell^{2}([a,b])$. $\mathcal{L}_{[a,b]}$ and $\mathcal{M}_{[a,b]}$ are defined similarly.

However, the matrix $\mathcal{E}_{[a,b]}$ will no longer be unitary due to the fact that $|\alpha_{a-1}|<1$ and $|\alpha_{b}|<1$. To solve this issue, we need to modify the boundary conditions briefly. With $\beta,\eta \in \partial \mathbb{D}$, define the sequence of Verblunsky coefficients
\begin{equation*}
\tilde{\alpha}_{n}=
\begin{cases}
\beta, \quad & n =a-1;\\
\eta,& n =b;\\
\alpha_{n},& n \notin  \{a-1,b\}.
\end{cases}
\end{equation*}
Denote the extended CMV matrix with Verblunsky coefficients $\tilde{\alpha}_{n}$ by $\tilde{\mathcal{E}}$. Define
\begin{equation*}
\mathcal{E}_{[a,b]}^{\beta,\eta}=P_{[a,b]}\tilde{\mathcal{E}}(P_{[a,b]})^{*}.
\end{equation*}
$\mathcal{L}_{[a,b]}^{\beta,\eta}$ and $\mathcal{M}_{[a,b]}^{\beta,\eta}$ are defined correspondingly. Then $\mathcal{E}_{[a,b]}^{\beta,\eta}$, $\mathcal{L}_{[a,b]}^{\beta,\eta}$ and $\mathcal{M}_{[a,b]}^{\beta,\eta}$ are all unitary.

For $z\in \mathbb{C}$, $\beta, \eta\in\partial \mathbb{D}$, we can define the characteristic determinant of matrix $\mathcal{E}^{\beta,\eta}_{[a,b]}$,
\begin{equation*}
\varphi^{\beta,\eta}_{[a,b]}(z):=\det(z-\mathcal{E}^{\beta,\eta}_{[a,b]}), \quad \phi^{\beta,\eta}_{[a,b]}(z):=(\rho_{a}\cdots\rho_{b})^{-1}\varphi^{\beta,\eta}_{[a,b]}(z).
\end{equation*}
Note that when $a>b$, $\phi^{\beta,\eta}_{[a,b]}(z)=1$.

Since the equation $\tilde{\mathcal{E}}u=zu$ is equivalent to $(z\mathcal{L}^{*}-\mathcal{M})u=0$,  the associated finite-volume Green's functions are as follows:
\begin{equation*}
G^{\beta,\eta}_{[a,b]}(z)=\big(z(\mathcal{L}^{\beta,\eta}_{[a,b]})^{*}-\mathcal{M}^{\beta,\eta}_{[a,b]}\big)^{-1},
\end{equation*}
\begin{equation*}
G^{\beta,\eta}_{[a,b]}(j,k;z)=\langle \delta_{j},G^{\beta,\eta}_{[a,b]}(z)\delta_{k}\rangle,\quad j,k\in [a,b].
\end{equation*}
According to \cite[Proposition 3.8]{Kruger13-IMRN} and \cite[Section B.1]{Zhu-JAT}, for $\beta, \eta\in\partial \mathbb{D}$, the Green's function has the expression:
\begin{equation*}
|G^{\beta,\eta}_{[a,b]}(j,k;z)|=\rho_{j}\cdots\rho_{k-1}\Big| \frac{\varphi^{\beta,\cdot}_{[a,j-1]}(z)\varphi^{\cdot,\eta}_{[k+1,b]}(z)}{\varphi^{\beta,\eta}_{[a,b]}(z)}\Big|,\quad a\leq j\leq k\leq b,
\end{equation*}
where ``$\cdot$" stands for the unchanged Verblunsky coefficient.

From \cite[Lemma 3.9]{Kruger13-IMRN}, if $u$ satisfies $\tilde{\mathcal{E}}u=zu$, Poisson's formula reads
\begin{align*}
u(m)=&G^{\beta,\eta}_{[a,b]}(a, m;z)
\begin{cases}
(z\overline{\beta}-\alpha_{a})u(a)-\rho_{a}u(a+1), \quad & a\text{ even}\\
(z\alpha_{a}-\beta)u(a)+z\rho_{a}u(a+1),& a \text { odd}
\end{cases}\\
&+G^{\beta,\eta}_{[a,b]}(m,b;z)
\begin{cases}
(z\overline{\eta}-\alpha_{b})u(b)-\rho_{b}u(b-1), \quad & b\text{ even}\\
(z\alpha_{b}-\eta)u(b)+z\rho_{b-1}u(b-1),& b \text { odd}
\end{cases}
\end{align*}
for $a<m<b$.

Without loss of generality, we restrict the extended CMV matrix to the interval $[0,n-1]\subset\mathbb{Z}$. Let $\varphi^{\beta,\eta}_{[0,n-1]}(\omega,z;x)=\det \big(z-\mathcal{E}^{\beta,\eta}_{[0,n-1]}\big)$ be the characteristic determinant of matrix $\mathcal{E}^{\beta,\eta}_{[0,n-1]}$. For $z\in \partial \mathbb{D}$, it follows from \cite[Theorem 2]{Wang-JMAA} that the relation between the $n$-step transfer matrix and the characteristic determinant is
\begin{equation}\label{relation}
M_n(\omega,z;x)=(\sqrt{z})^{-n}\Big(\prod_{j=0}^{n-1}\frac{1}{\rho_j}\Big)
\left(
\begin{matrix}
z\varphi^{\beta,\eta}_{[1,n-1]} & \frac{z\varphi^{\beta,\eta}_{[1,n-1]}-\varphi^{\beta,\eta}_{[0,n-1]}}{\alpha_{-1}}\\
z\Big(\frac{z\varphi^{\beta,\eta}_{[1,n-1]}-\varphi^{\beta,\eta}_{[0,n-1]}}{\alpha_{-1}}\Big)^* &\big(\varphi^{\beta,\eta}_{[1,n-1]}\big)^{*}
\end{matrix}
\right).
\end{equation}

Based upon the above analysis, we can get the Green's function estimate which will be of significance in the proof of AL.
\begin{lemma}\label{Green-estimate-element}
Assume $L(z)>\gamma$. Then for $n\gg N_{0}(\gamma)$, there exists a set $\Omega\subset\mathbb{T}^{d}$ satisfying
\begin{equation}\label{Green-estimate-element(1)}
\mathrm{mes}\Omega<\exp(-C_{0}n^{\sigma}),
\end{equation}
where $C_{0}$ and $\sigma$ are as in LDT. Furthermore, for any $x$ outside $\Omega$,
\begin{equation}\label{Green-estimate-element(2)}
|G_{\Lambda}^{\beta,\eta}(j,k;z)|<e^{-L(z)|j-k|+n^{1-}}
\end{equation}
holds, where $\Lambda$ is one of
$$\{[0,n-1], [1,n-1]\}.$$
\end{lemma}
\begin{proof}
Without loss of generality, we assume that $0\leq j\leq k\leq n-1$ and $n\gg N_{0}(\gamma)$. Then one can obtain that
\begin{align*}
|G^{\beta,\eta}_{[0,n-1]}(j,k;z)|&=\rho_{j}\cdots\rho_{k-1}\Big| \frac{\varphi^{\beta,\cdot}_{[0,j-1]}(z)\varphi^{\cdot,\eta}_{[k+1,n-1]}(z)}{\varphi^{\beta,\eta}_{[0,n-1]}(z)}\Big|\\
&\leq\Big| \frac{\varphi^{\beta,\cdot}_{[0,j-1]}(z)\varphi^{\cdot,\eta}_{[k+1,n-1]}(z)}{\varphi^{\beta,\eta}_{[0,n-1]}(z)}\Big|\\
&\lesssim \frac{\|M_{j}(\omega,z;x)\|\|M_{n-k}(\omega,z;T^{k}x)\|}{\|M_{n}(\omega,z;x)\|}
\end{align*}
provided that $\varphi^{\beta,\eta}_{[0,n-1]}(z)\gtrsim \|M_{n}(\omega,z;x)\|$.

Let $\Omega$ be the set of LDT. That is to say,
$$\Omega=\{x\in \mathbb{T}^{d}:\big|\log\|M_{n}(\omega,z;x)\|-nL_{n}(\omega,z)\big|>n^{1-\tau}\}.$$
According to the relation between the characteristic determinant and the $n$-step transfer matrix, for $x\notin\Omega$, we get that
$$e^{nL_{n}(\omega,z)-n^{1-\tau}}<\|M_{n}(\omega,z;x)\|<e^{nL_{n}(\omega,z)+n^{1-\tau}}.$$
From Lemma \ref{Lemma3.4}, we have that
$$L_{n}(\omega,z)<L(\omega,z)+C\frac{(\log n)^{1/\sigma}}{n}.$$
Thus,
$$e^{nL(\omega,z)+C(\log n)^{1/\sigma}-n^{1-\tau}}<\|M_{n}(\omega,z;x)\|<e^{nL(\omega,z)+C(\log n)^{1/\sigma}+n^{1-\tau}}.$$
As a consequence, we obtain that
\begin{align*}
|G^{\beta,\eta}_{[0,n-1]}(j,k;z)|& \lesssim\frac{\|M_{j}(\omega,z;x)\|\|M_{n-k}(\omega,z;T^{k}x)\|}{\|M_{n}(\omega,z;x)\|}\\
&\leq\frac{e^{(n-(k-j))L(z)+2C(\log n)^{1/\sigma}+n^{1-\tau}}}{e^{nL(z)+C(\log n)^{1/\sigma}-n^{1-\tau}}}\\
&\leq e^{-L(z)|j-k|+n^{1-}}.
\end{align*}
\end{proof}

\subsection{Semialgebraic Sets}
In this section, we recall some basic preliminaries about the semialgebraic sets developed by Bourgain, which are extremely powerful tools to prove the main theorem in the present paper.
\begin{definition 2}\cite[Definition 9.1]{Bourgain-book}\label{semialgebraic-set}
A set $\mathcal{S}\subset \mathbb{R}^{n}$ is called semialgebraic if it is a finite union of sets defined by a finite number of polynomial equalities and inequalities. More precisely, let $\mathcal{P}=\{P_{1},\ldots,P_{s}\}\subset R[X_{1},\ldots,X_{n}]$ be a family of real polynomials whose degrees are bounded by $d$. A (closed) semialgebraic set $\mathcal{S}$ is given by an expression
\begin{equation}\label{semialgebraic(1)}
\mathcal{S}=\bigcup\limits_{j}\bigcap\limits_{l\in \mathcal{L}_{j}}\{R^{n}|P_{l}s_{jl}0\},
\end{equation}
where $\mathcal{L}_{j}\subset\{1,\ldots,s\}$ and $s_{jl}\in \{\geq,\leq,=\}$ are arbitrary. We say that $\mathcal{S}$ has degree at most $sd$, and its degree is the infimum of $sd$ over all representations as in \eqref{semialgebraic(1)}.
\end{definition 2}

\begin{lemma}\cite[Proposition 9.2]{Bourgain-book}\label{Prop9.2}
Let $\mathcal{S}\subset \mathbb{R}^{n}$ be semialgebraic defined in terms of $s$ polynomials of degree at most $d$ as in \eqref{semialgebraic(1)}. Then there exists a semialgebraic description of its projection onto $\mathbb{R}^{n-1}$ by a formula involving at most $s^{2n}d^{O(n)}$ polynomials of degree at most $d^{O(n)}$. In particular, if $\mathcal{S}$ has degree $B$, then any projection of $\mathcal{S}$ has degree at most $B^{C}$, $C=C(n)$.
\end{lemma}

\begin{lemma}\cite[Corollary 9.7]{Bourgain-book}\label{Coro9.7}
Let $\mathcal{S}\subset[0,1]^{n}$ be semialgebraic of degree $B$ and $\mathrm{mes}_{n}\mathcal{S}<\eta$. Let $\omega\in \mathbb{T}^{n}$ satisfy Diophantine condition and $N$ be a large integer,
$$\log B\ll \log N<\log\frac{1}{\eta}.$$
Then for any $x_{0}\in \mathbb{T}^{n}$,
\begin{equation}\label{Coro9.7-(1)}
\#\{k=1,\ldots,N| x_{0}+k\omega\in \mathcal{S}(\mathrm{mod}\,1)\}<N^{1-\delta}
\end{equation}
for some $\delta=\delta(\omega)$.
\end{lemma}

\begin{lemma}\cite[Lemma 9.9]{Bourgain-book}\label{Lemma9.9}
Let $\mathcal{S}\subset[0,1]^{2n}$ be semialgebraic of degree $B$ and $\mathrm{mes}_{2n}\mathcal{S}<\eta$, $\log B\ll \log \frac{1}{\eta}$. We denote $(\omega,x)\in [0,1]^{n}\times [0,1]^{n}$ the product variable. Fix $\varepsilon>\eta^{\frac{1}{2n}}$. Then there is a decomposition
$$\mathcal{S}=\mathcal{S}_{1}\cup\mathcal{S}_{2}$$
with $\mathcal{S}_{1}$ satisfying
\begin{equation}\label{Lemma9.9-(1)}
\mathrm{mes}(\mathrm{Proj}_{\omega}\mathcal{S}_{1})<B^{C}\varepsilon
\end{equation}
and $\mathcal{S}_{2}$ satisfying the transversality property
\begin{equation}\label{Lemma9.9-(2)}
\mathrm{mes}(\mathcal{S}_{2}\cap L)<B^{C}\varepsilon^{-1}\eta^{1/2n}
\end{equation}
for any $n$-dimensional hyperplane $L$ such that $\mathop{\max}\limits_{0\leq j\leq n-1}|\mathrm{Proj}_{L}(e_{j})|<\frac{1}{100}\varepsilon$ (we denote $(e_{0},\ldots,e_{n-1})$ the $\omega$-coordinate vectors).
\end{lemma}

\section{Proof of Theorem \ref{main-theorem}}\label{sec:proof}

The proof employs a multi-scale analysis based on Green's function estimates and double resonance exclusion.

The core strategy is to show that for sufficiently large scale $n_0$, the finite-volume Green's function $G_{[0,n_0-1]}^{\beta,\eta}$ exhibits exponential off-diagonal decay for most phases $x$. This decay property is the key indicator of localization and is formalized as:

\begin{equation}\label{AL-(1)}
|G_{[0,n_{0}-1]}^{\beta,\eta}(m_{1},m_{2};z)|<e^{-\gamma |m_{1}-m_{2}|+n_{0}^{1-}}
\end{equation}
for all $0\leq m_{1}, m_{2}\leq n_{0}-1$. To operationalize this condition, we replace \eqref{AL-(1)}  with an equivalent determinant-based criterion that is more suitable for semi-algebraic analysis:
\begin{align}
\mathop{\sum}\limits_{0\leq m_{1}, m_{2}\leq n_{0}-1}&e^{2\gamma|m_{1}-m_{2}|}\big|\mathrm{det}[(m_{1},m_{2})-\mathrm{minor}\,\, \mathrm{of}\,\, (z-\mathcal{E}_{[0,n_{0}-1]}^{\beta,\eta}(x,\omega)) \big|^{2} \nonumber\\
&<e^{2n_{0}^{1-}}\big|\mathrm{det}(z-\mathcal{E}_{[0,n_{0}-1]}^{\beta,\eta}(x,\omega))\big|^{2}.\label{AL-(2)}
\end{align}
This reformulation is crucial as it expresses the Green's function decay in terms of polynomial conditions on the matrix entries, enabling the application of semi-algebraic geometry tools.

The remainder of the proof will establish that condition \eqref{AL-(2)} holds for most phases $x$ and frequencies $\omega$, outside a set of exponentially small measure.

Fix $n_0$ and consider the property \eqref{AL-(2)}. The left-hand side of \eqref{AL-(2)} is a polynomial in $(\operatorname{Re}\alpha(x),\operatorname{Im}\alpha(x),\operatorname{Re} z,\operatorname{Im} z)$. Since $\Re\alpha(x)$ and $\Im \alpha(x)$ are real analytic functions,  we can replace them by the trigonometric polynomials
$$f(x)=\mathop{\sum}\limits_{|k|\leq C n_{0}} f_{k}e^{i\langle k,x \rangle} \quad \mathrm{and}\quad g(x)=\mathop{\sum}\limits_{|k|\leq C n_{0}} g_{k}e^{i\langle k,x \rangle}$$
respectively with error less than $e^{-\tilde{C}n_{0}^{2}}$ ($\tilde{C}$ a constant), where $|f_{k}|, |g_{k}|<\exp(-r|k|)$, $r$ is a constant, $C$ is a sufficiently large constant. Then we can obtain a polynomial inequality
\begin{equation}\label{AL-(3)}
P(\cos \omega_1,\cdots, \cos \omega_d, \sin \omega_1, \cdots, \sin \omega_d,\cos x_1, \cdots, \cos x_d,\sin x_1, \cdots, \sin x_d,\Re z, \Im z)>0
\end{equation}
to replace the condition \eqref{AL-(2)}, where the degree of
$$P(\cos \omega_1,\cdots, \cos \omega_d, \sin \omega_1, \cdots, \sin \omega_d, \cos x_1, \cdots, \cos x_d,\sin x_1, \cdots, \sin x_d, \Re z, \Im z)$$
is at most $C_{1}n_{0}^{2}$. One can further truncate the Taylor series of the trigonometric functions and replace \eqref{AL-(3)} by a polynomial inequality
\begin{equation}\label{AL-(4)}
P(\omega, x, \Re z, \Im z)>0
\end{equation}
whose degree is at most $C_{2}n_{0}^{3}$.

Fix $\omega\in \mathbb{T}^{d}(p,q)$ and $z \in \mathcal{I_{\alpha}}$. According to Lemma \ref{Green-estimate-element}, the exceptional set $\Omega$ does not only satisfy
\begin{equation}\label{AL-(5)}
\mathrm{mes}\Omega<\exp(-C_{0}n^{\sigma})
\end{equation}
but also may be assumed semialgebraic of degree less than $Cn_{0}^{3}$. Notice that this set depends on $z$ and $\omega$.

Fix $n=n_{0}$ in Lemma \ref{Green-estimate-element} and redefine $\Lambda$ to be one of the intervals
\begin{equation}\label{AL-(6)}
\{[-n_{0}+1,n_{0}-1], [-n_{0},n_{0}-1]\}.
\end{equation}
Let $\Omega=\Omega(z)$ be as above. For $x\notin \Omega$, one of the intervals $\Lambda$ (depending on $x$) satisfies
\begin{equation}\label{AL-(7)}
|G_{\Lambda}^{\beta,\eta}(m_{1},m_{2};z)|<e^{-\gamma|m_{1}-m_{2}|+n_{0}^{1-}}
\end{equation}
for all $m_{1}, m_{2}\in \Lambda$.

Fix $x_{0}\in \mathbb{T}^{d}$. Now we consider the orbit $\{x_{0}+j\omega:|j|\leq n_{1}\}$. In this case, we let $n_{1}=n_{0}^{C}$, where $C$ is a sufficiently large constant.

Apply Lemma \ref{Coro9.7} with $\mathcal{S}=\Omega$, $B=C_{2}n_{0}^{3}$, $\eta=\exp(-C_{0}n^{\sigma})$ and $n=n_{1}$. Then the statement \eqref{Coro9.7-(1)} implies that except for at most $n_{1}^{1-\delta}$ values of $|j|<n_{1}$, taking $x=x_{0}+j\omega$, one of the intervals $\Lambda$ from \eqref{AL-(6)} satisfies \eqref{AL-(7)}.

Assume that $u=\{u(n)\}_{n\in \mathbb{Z}}$ satisfies the equation
$$\tilde{\mathcal{E}}(x_{0},\omega)u=zu$$
with $u(0)=1$ and $|u(n)|\lesssim n^{C}$, i.e., $z$ and $u$ are a generalized eigenvalue and a generalized eigenfunction of  $\tilde{\mathcal{E}}(x_{0},\omega)$,  respectively.

If $\Lambda+j=[a,b]$, then $\mathcal{E}_{\Lambda+j}=R_{\Lambda+j}\tilde{\mathcal{E}}R_{\Lambda+j}^{*}$. According to the Poisson's formula,
\begin{align*}
u(n)=&G^{\beta,\eta}_{[a,b]}(a, n;z)
\begin{cases}
(z\overline{\beta}-\alpha_{a})u(a)-\rho_{a}u(a+1), \quad & a\text{ even}\\
(z\alpha_{a}-\beta)u(a)+z\rho_{a}u(a+1),& a \text { odd}
\end{cases}\\
&+G^{\beta,\eta}_{[a,b]}(n,b;z)
\begin{cases}
(z\overline{\eta}-\alpha_{b})u(b)-\rho_{b}u(b-1), \quad & b\text{ even}\\
(z\alpha_{b}-\eta)u(b)+z\rho_{b-1}u(b-1),& b \text { odd}
\end{cases}
\end{align*}
for $a<n<b$. Therefore,
\begin{align*}
|u(n)|&\leq n^{C}|G^{\beta,\eta}_{[a,b]}(a, n;z)|+n^{C}|G^{\beta,\eta}_{[a,b]}(n,b;z)|\\
&\leq n^{C}e^{-\gamma |n-a|+n_{0}^{1-}}+n^{C}e^{-\gamma |b-n|+n_{0}^{1-}}\\
&\leq n^{C}e^{n_{0}^{1-}}(e^{-\gamma|n-a|}+e^{-\gamma|b-n|}).
\end{align*}
In particular, taking $n=j$, we have that $|j-a|>\frac{n_{0}}{2}$ and $|j-b|>\frac{n_{0}}{2}$. Then it follows that
\begin{equation}\label{AL-(8)}
|u(j)|<e^{-\frac{\gamma}{2}n_{0}}
\end{equation}
holds except at most $n^{1-\delta}$ values of $|j|<n_{1}$.

Next, let $I=[-j_{0}+1,j_{0}-1]$. Then $\mathcal{E}_{I}=R_{I}\tilde{\mathcal{E}}R_{I}^{*}$. Hence,
\begin{align*}
1=|u(0)|\leq&|G^{\beta,\eta}_{I}(-j_{0}+1, 0;z)|(2|u(-j_{0}+1)|+|u(-j_{0}+2)|)\\
&+|G^{\beta,\eta}_{I}(0,j_{0}-1;z)|(2|u(j_{0}-1)|+|u(j_{0}-2)|).
\end{align*}
If $-j_{0}+1$, $-j_{0}+2$, $j_{0}-2$, $j_{0}-1$ satisfy \eqref{AL-(8)}, one can obtain that
\begin{equation}\label{AL-(9)}
\|G_{[-j_{0}+1,j_{0}-1]}^{\beta,\eta}(x_{0},z)\|\geq \frac{1}{6}e^{\frac{\gamma}{2}n_{0}}.
\end{equation}
Equivalently,
\begin{equation}\label{AL-(10)}
\mathrm{dist}(z,\sigma(\mathcal{E}_{[-j_{0}+1,j_{0}-1]}^{\beta,\eta}(x_{0})))<6e^{-\frac{\gamma}{2}n_{0}}.
\end{equation}
Therefore, fixing $z \in \mathcal{I_{\alpha}}$, if there is a state $u$ with $u(0)=1$, then for any large $n_{0}$, there exists some $j_{0}$, $|j_{0}|<n_{1}=n_{0}^{C}$ for which \eqref{AL-(10)} holds.

Denote $\Sigma_{\omega}=\mathop{\cup}\limits_{|j|\leq n_{1}}\sigma(\mathcal{E}_{[-j,j]}^{\beta,\eta}(x_{0}))$. It follows from \eqref{AL-(7)} and \eqref{AL-(10)} that if
\begin{equation}\label{AL-(11)}
x\notin \mathop{\cup}\limits_{z'\in \Sigma_{\omega}}\Omega(z'),
\end{equation}
then one of the sets in \eqref{AL-(6)} satisfies
\begin{equation}\label{AL-(12)}
|G_{\Lambda}^{\beta,\eta}(m_{1},m_{2};z)|<e^{-\gamma|m_{1}-m_{2}|+n_{0}^{1-}}
\end{equation}
for $m_{1}, m_{2}\in \Lambda$.

Now consider the interval $[\frac{n_{2}}{2},2n_{2}]$ with $n_{2}=n_{0}^{C'}$ ($C'$ a sufficiently large constant). Suppose that we ensured that
\begin{equation}\label{AL-(13)}
x_{0}+n\omega\notin \mathop{\cup}\limits_{z'\in \Sigma_{\omega}}\Omega(z')(\mathrm{mod}\, 1)\quad \mathrm{for}\,\,\mathrm{all}\,\, \frac{n_{2}}{2}<|n|<2n_{2}.
\end{equation}
Thus, for each $\frac{n_{2}}{2}<|n|<2n_{2}$, there is an interval
$$\Lambda^{(n)}\in\{[-n_{0}+1,n_{0}-1], [-n_{0},n_{0}-1]\}$$
for which \eqref{AL-(12)} holds:
\begin{equation}\label{AL-(14)}
|G_{\Lambda^{(n)}+n}^{\beta,\eta}(m_{1},m_{2};z)|<e^{-\gamma|m_{1}-m_{2}|+n_{0}^{1-}}\,\, \mathrm{for}\,\, m_{1},\, m_{2}\in \Lambda^{(n)}+n.
\end{equation}
Define the interval
$$\tilde{\Lambda}=\mathop{\cup}\limits_{\frac{n_{2}}{2}<n<2n_{2}}(\Lambda^{(n)}+n)\supset [\frac{n_{2}}{2},2n_{2}].$$
By the paving property from \cite[Page 839]{BG00-Annals} or \cite[Appendix 7.2]{WD19-JFA}, one can obtain that
\begin{equation}\label{AL-(15)}
|G_{\tilde{\Lambda}}^{\beta,\eta}(m_{1},m_{2};z)|<e^{-\frac{\gamma}{2}|m_{1}-m_{2}|+n_{0}^{1-}}\,\, \mathrm{for}\,\, m_{1},\, m_{2}\in [\frac{n_{2}}{2},2n_{2}].
\end{equation}
Therefore, for $k\in [\frac{n_{2}}{2},2n_{2}]$, we get that
\begin{align*}
|u(k)|&\leq |G_{[\frac{n_{2}}{2},2n_{2}]}^{\beta,\eta}(\frac{n_{2}}{2},k;z)|n_{2}^{C}
+|G_{[\frac{n_{2}}{2},2n_{2}]}^{\beta,\eta}(k,2n_{2};z)|n_{2}^{C}\\
&<n_{2}^{C}e^{-\frac{\gamma}{2}|k-\frac{n_{2}}{2}|}+n_{2}^{C}e^{-\frac{\gamma}{2}|k-2n_{2}|}\\
&<e^{-\frac{\gamma}{2}k}.
\end{align*}
For $k\in \mathbb{Z}_{-}$, we can get the same result. This implies that the exponential decay property follows.

Now we verify the assumption \eqref{AL-(13)}.

For $|j|\leq n_{1}$, we consider the set
\begin{equation*}
\mathcal{B}=\{(\omega,z',x):\omega\in \mathbb{T}^{d}(p,q), z'\in \sigma(\mathcal{E}_{[-j,j]}^{\beta,\eta}(x_{0})), x\in \Omega(z')\}\subset \mathbb{T}^{d}\times \partial\mathbb{D}\times\mathbb{T}^{d}.
\end{equation*}
For fixed $\omega$, the condition $z'\in\sigma(\mathcal{E}_{[-j,j]}^{\beta,\eta}(x_{0}))$ and $x\in \Omega(z')$ can be replaced by the inequalities of the polynomials respectively as before. Therefore, $\mathcal{B}$ is a semialgebraic set of degree at most $C_{3}n_{1}^{3}$. Let $\mathcal{S}=\mathrm{Proj}_{(\omega,x)}\mathcal{B}\subset\mathbb{T}^{d}\times \mathbb{T}^{d}$. Obviously, $\mathrm{mes}\mathcal{S}<\exp(-C_{0}n_{1}^{\sigma})$. According to Lemma \ref{Prop9.2}, $\mathcal{S}$ is a semialgebraic set whose degree is at most $n_{1}^{3C_{4}}$ for some constant $C_{4}$. Returning to assumption \eqref{AL-(13)}, we need to verify that
\begin{equation}\label{AL-(16)}
(\omega,x_{0}+n\omega)\notin \mathcal{S}
\end{equation}
for $\frac{n_{2}}{2}\leq |n|\leq 2n_{2}$. Now we apply Lemma \ref{Lemma9.9} with $n=d$, $B=n_{1}^{3C_{4}}$, $\eta=\exp(-C_{0}n_{1}^{\sigma})$. Take $\varepsilon=n_{2}^{-\frac{1}{10}}$. Then there is a decomposition $\mathcal{S}=\mathcal{S}_{1}\cup\mathcal{S}_{2}$. It follows that
\begin{equation}\label{Al-(17)}
\mathrm{mes}(\mathrm{Proj}_{\omega}\mathcal{S}_{1})<B^{C}\varepsilon<n_{1}^{3C_{4}C}n_{2}^{-\frac{1}{10}}<n_{2}^{-\frac{1}{11}},
\end{equation}
where $n_{2}$ is large enough. Do the partition
$$[0,1]^{d}=\mathop{\bigcup}\limits_{l}\Big(x_{l}+[0,\frac{1}{n_{2}}]^{d}\Big),\,\,l\leq n_{2}^{d}.$$
For fixed $\frac{n_{2}}{2}\leq |n|\leq 2n_{2}$ and $l$, we consider $L=\Big\{(\omega,x_{0}+nx_{l}+n\omega):\omega\in [0,\frac{1}{n_{2}}]^{d}\Big\}$,
which is a translate of the $d$-hyperplane $\Big\{\frac{1}{n}e_{j}+e_{j+d}:0\leq j \leq d-1\Big\}.$
Furthermore, we have that $\mathop{\max}\limits_{0\leq j<d}|\mathrm{Proj}_{L}e_{j}|<\frac{1}{100}\varepsilon<\varepsilon^{2}$. By Lemma \ref{Lemma9.9}, one can obtain that
\begin{small}
$$\mathrm{mes}_{\mathbb{T}^{d}}\Big\{\omega\in [0,\frac{1}{n_{2}}]^{d}:(\omega,x_{0}+nx_{l}+n\omega)\in\mathcal{S}_{2}\Big\}=\mathrm{mes}(\mathcal{S}_{2}\cap L)
<B^{C}\varepsilon^{-1}\eta^{\frac{1}{2d}}
<n_{1}^{3C_{4}C}n_{2}^{\frac{1}{10}}\exp(-\frac{C_{0}}{2d}n_{1}^{\sigma}).$$
\end{small}
Summing the contributions over $n$ and $l$, we have that
\begin{align}
\mathrm{mes}&\Big\{\omega:(\omega,x_{0}+n\omega)\in \mathcal{S}_{2}\,\,\mathrm{for}\,\,\mathrm{some}\,\,\frac{1}{n_{2}}<|n|<2n_{2}\Big\}\nonumber\\
&<n_{2}^{d+1}n_{1}^{3C_{4}C}n_{2}^{\frac{1}{10}}\exp(-\frac{C_{0}}{2d}n_{1}^{\sigma})<\exp(-n_{1}^{\frac{\sigma}{2}}).\label{AL-(18)}
\end{align}
From \eqref{Al-(17)} and \eqref{AL-(18)}, we can get an $\omega$-set whose measure is at most
$$n_{2}^{-\frac{1}{11}}+\exp(-n_{1}^{\frac{\sigma}{2}})<n_{2}^{-\frac{1}{12}}.$$
Summing these sets over $j$, $j\leq n_{1}$, we get an $\omega$-set of measure at most $n_{1}n_{2}^{-\frac{1}{12}}<n_{2}^{-\frac{1}{13}}$. This implies that \eqref{AL-(13)} holds for $\omega$ outside the $\omega$-set we obtained.

Since we initially fixed $n_{0}$ and set $n_{1}=n_{0}^{C}$, $n_{2}=n_{0}^{C'}$, we can denote the above $\omega$-set to be $\Omega_{n_{0}}$. Then we have that $$\mathrm{mes}\Omega_{n_{0}}<n_{2}^{-\frac{1}{13}}=n_{0}^{-\frac{C'}{13}}<n_{0}^{-10}.$$

According to the above analysis, we ensure that if
$$\mathcal{E}(x_{0},z)u=zu,\,\, u(0)=1,\,\, |u(n)|\lesssim n^{C},$$
then
$$|u(k)|<e^{-\frac{\gamma}{4}k}\quad \mathrm{for}\,\,\, k\in\big[\frac{n_{2}}{2},2n_{2}\big].$$
Let $\Omega_{*}=\mathop{\cup}\limits_{n}\mathop{\cap}\limits_{n_{0}>n}\Omega_{n_{0}}$. It follows that $\mathrm{mes}\Omega_{*}=0$ and AL holds for $\omega\in \mathbb{T}^{d}(p,q)\backslash \Omega_{*}$.

$\hfill \Box$

\section{Application to Quantum Walks}\label{sec:appl}

\subsection{Quantum Walks and CMV Matrices}\label{QWCMV}
Quantum walks share a structural similarity with CMV matrices via unitary equivalence. Now we recall some basic knowledge about quantum walks, see \cite[Section 2.4]{DFO16-JMPA}.

A quantum walk is described by a unitary operator on the Hilbert space $\mathcal{H}=\ell^{2}(\mathbb{Z})\otimes \mathbb{C}^{2}$, which models a state space in which a wave packet comes equipped with a ``spin" at each integer site. Notice that the elementary tensors of the form $\delta_{n}\otimes e_{\uparrow}$ and $\delta_{n}\otimes e_{\downarrow}$ with $n\in \mathbb{Z}$ comprise an orthonormal basis of $\mathcal{H}$. Here, $\{e_{\uparrow},e_{\downarrow}\}$ denotes the canonical basis of $\mathbb{C}^{2}$. A time-homogeneous quantum walk scenario is given as soon as coins
\begin{equation}\label{coin}
C_{n}=
\left(
\begin{array}{cc}
c_{n}^{11}&c_{n}^{12}\\
c_{n}^{21}&c_{n}^{22}
\end{array}
\right)\in\mathbb{U}(2), \quad n\in \mathbb{Z}
\end{equation}
are specified. Assume that $c_{n}^{11}, c_{n}^{22}\neq 0$ and $\mathrm{det}\,C_{n}=1$. As one passes from time $t$ to $t+1$, the update rule of the quantum walk is
\begin{align*}
&\delta_{n}\otimes e_{\uparrow}\mapsto c_{n}^{11}\delta_{n+1}\otimes e_{\uparrow}+c_{n}^{21}\delta_{n-1}\otimes e_{\downarrow},\\
&\delta_{n}\otimes e_{\downarrow}\mapsto c_{n}^{12}\delta_{n+1}\otimes e_{\uparrow} +c_{n}^{22} \delta_{n-1}\otimes e_{\downarrow}.
\end{align*}
If we extend this by linearity and continuity to general elements of $\mathcal{H}$, this defines a unitary operator $U$ on $\mathcal{H}$. For a typical element $\psi\in \mathcal{H}$, we may describe the action of $U$ in coordinates via
\begin{align*}
&[U\psi]_{\uparrow,n}=c_{n-1}^{11}\psi_{\uparrow,n-1}+c_{n-1}^{12}\psi_{\downarrow,n-1},\\
&[U\psi]_{\downarrow,n}=c_{n+1}^{21}\psi_{\uparrow,n+1}+c_{n+1}^{22}\psi_{\downarrow,n+1}.
\end{align*}
Then the matrix representation of $U$ is
\begin{equation}\label{matrix-U}
U=
\left(
\begin{array}{ccccccccc}
\ddots&\ddots&\ddots&\ddots&  &  & & & \\
& 0&0&c_{0}^{21}&c_{0}^{22} & & & &\\
& c_{-1}^{11} &c_{-1}^{12}& 0&0& & & & \\
& & & 0 & 0 & c_{1}^{21}& c_{1}^{22} & & \\
& & & c_{0}^{11} & c_{0}^{12} & 0& 0 & & \\
& & & & &0&0 &c_{2}^{21}&c_{2}^{22}\\
& & & & &c_{1}^{11}&c_{1}^{12} &0&0\\
& & & & &\ddots&\ddots &\ddots&\ddots
\end{array}
\right).
\end{equation}

Comparing the above matrix with the extended CMV matrix, the structures of these two matrices are similar. If all Verblunsky coefficients with even index vanish, then the extended CMV matrix becomes
\begin{equation}\label{CMV-vanish}
\mathcal{E}=
\left(
\begin{array}{ccccccccc}
\ddots&\ddots&\ddots&\ddots&  &  & & & \\
& 0&0&\bar{\alpha}_{1}&\rho_{1} & & & &\\
& \rho_{-1}&-\alpha_{-1}& 0&0& & & & \\
& & & 0 & 0 &\bar{\alpha}_{3} & \rho_{3} & & \\
& & & \rho_{1} & -\alpha_{1} & 0& 0 & & \\
& & & & &0&0 &\bar{\alpha}_{5}&\rho_{5}\\
& & & & &\rho_{3}&-\alpha_{3} &0&0\\
& & & & &\ddots&\ddots &\ddots&\ddots
\end{array}
\right).
\end{equation}
The matrix in \eqref{CMV-vanish} strongly resembles the matrix representation of $U$ in \eqref{matrix-U}. However, $\rho_{n}>0$ for all $n$, then \eqref{matrix-U} and \eqref{CMV-vanish} may not match exactly when $c_{n}^{kk}$ is not real and positive. Indeed, this can be easily resolved by conjugation with a suitable diagonal unitary, as shown in
\cite{CMV-quantum-walks}.

Given $U$ as in \eqref{matrix-U}, write
$$c_{n}^{kk}=r_{n}\omega_{n}^{k}, \,\,\, n\in \mathbb{Z},\,\,\,k\in \{1,2\},\,\,\,r_{n}>0,\,\,\,|\omega_{n}^{k}|=1.$$
Define $\{\lambda_{n}\}_{n\in \mathbb{Z}}$ as
$$\lambda_{0}=1,\,\,\,\lambda_{-1}=1,\,\,\,\lambda_{2n+2}=\omega_{n}^{1}\lambda_{2n},\,\,\,\lambda_{2n+1}=\overline{\omega_{n}^{2}}\lambda_{2n-1}.$$
Let $\mathcal{D} =\mathrm{diag}(\cdots,\lambda_{-1},\lambda_{0},\lambda_{1},\cdots)$. Then we can obtain that
\begin{equation}\label{eq:unitary-equiv}
\hat{\mathcal{E}}=\mathcal{D}^{*}U\mathcal{D},
\end{equation}
where $\hat{\mathcal{E}}$ is the extended CMV matrix whose Verblunsky coefficients are defined by
\begin{equation}\label{new-Verblunsky}
\hat{\alpha}_{2n}=0,\,\,\,\hat{\alpha}_{2n+1}=\frac{\lambda_{2n-1}}{\lambda_{2n}}\cdot \overline{c_{n}^{21}}=-\frac{\lambda_{2n+1}}{\lambda_{2n+2}}\cdot c_{n}^{12},\,\,\, n\in \mathbb{Z}.
\end{equation}
Notice that the hypotheses $c_{n}^{11}, c_{n}^{22}\neq 0$ imply $\rho_{n}>0$ for all $n\in \mathbb{Z}$. According to the analysis above, there is a close relation between a quantum walk and an extended CMV matrix whose Verblunsky coefficients with even index are zero. Moreover, it is worth mentioning that in \cite{Cedzich-CMP-2023, Cedzich-IMRN-2024}, Cedzich et al. showed that any extended CMV matrix is gauge-equivalent to a quantum walk.\footnote{Our gauge equivalence covers general split-step quantum walks as defined in \cite{Cedzich-CMP-2023, Cedzich-IMRN-2024}}.

\subsection{Gesztesy-Zinchenko Cocycle}\label{G-Z-C}
Now we consider the case of coined quantum walks on the integer lattice where the coins are distributed quasi-periodically, i.e.,
\begin{equation*}
C_{n}=C_{n,x,\omega}=
\left(
\begin{array}{cc}
c_{n}^{11}(x+n\omega)&c_{n}^{12}(x+n\omega)\\
c_{n}^{21}(x+n\omega)&c_{n}^{22}(x+n\omega)
\end{array}
\right)
\end{equation*}
with analytic sampling functions and frequency vector $\omega\in \mathbb{T}^{d}(p,q)$.  Here we also assume that $\mathrm{det}\,C_{n,x,\omega}=1$. We denote the corresponding quantum walk by $U_{\omega}(x)$.

As for the extended CMV matrix $\hat{\mathcal{E}}$, there is a decomposition $\hat{\mathcal{E}}=\hat{\mathcal{L}}\hat{\mathcal{M}}$ (see Section \ref{3.4} for details). For a given solution $u$ of $\hat{\mathcal{E}}u=zu$, we define $v=\hat{\mathcal{L}}^{-1}u$, then $\hat{\mathcal{M}}u=zv$. According to \cite[Section 2]{FOV18-JMAA}, we have that
\begin{equation*}
\left(
\begin{array}{c}
u(n+1)\\
v(n+1)
\end{array}
\right)
=Y(n,z)
\left(
\begin{array}{c}
u(n)\\
v(n)
\end{array}
\right),
\end{equation*}
where
\begin{equation*}
Y(n,z)=\frac{1}{\rho_{n}}
\begin{cases}
\left(
\begin{array}{cc}
-\hat{\alpha}_{n}&1\\
1&-\overline{\hat{\alpha}}_{n}
\end{array}
\right)\quad  n\,\, \text{is even},\\ \\
\left(
\begin{array}{cc}
-\overline{\hat{\alpha}}_{n}&z\\
z^{-1}&-\hat{\alpha}_{n}
\end{array}
\right)\quad  n\,\, \text {is odd}.
\end{cases}\\
\end{equation*}
For the sake of convenience, we write $Y(n,z)=P(\hat{\alpha}_{n},z)$ when $n$ is even and $Y(n,z)=Q(\hat{\alpha}_{n},z)$ when $n$ is odd. According to \cite[Section 3]{DFLY}, we can define Gesztesy-Zinchenko cocycle by
$$G(\omega,z;x)=Q(\hat{\alpha}(x+\omega),z)P(\hat{\alpha}(x),z).$$
Based upon the above discussion, if the Verblunsky coefficients with even index are equal to zero, then the Gesztesy-Zinchenko cocycle is
\begin{equation*}
G(\omega,z;x)=\frac{1}{\rho(x)}
\left(
\begin{array}{cc}
z&-\overline{\hat{\alpha}}(x)\\
-\hat{\alpha}(x)&z^{-1}
\end{array}
\right).
\end{equation*}
Correspondingly, the $n$-step transfer matrix is
$$G_{n}(\omega,z;x)=\prod_{j=n-1}^{0}G(\omega,z;x+j\omega).$$
Therefore, the associated Lyapunov exponent of the quantum walk $U_{\omega}(x)$ can be defined by
$$L^{U}(w,z):=\lim_{n\rightarrow\infty}\frac{1}{n}\int_{\mathbb{T}^{d}}\log\|G_{n}(\omega,z;x)\|dx.$$

To define the Lyapunov exponent of quantum walks, one can also consider other transfer matrices instead of Gesztesy-Zinchenko cocycle; for more details, see \cite{Cedzich-JST-2024,Cedzich-IMRN-2024}. Indeed, all these cocycles are equivalent.
\begin{remark}\rm
For split-step quantum walks, the transfer matrix in \cite[Eq.(2.4)]{Cedzich-IMRN-2024} provides a non-alternating representation, but our results hold under any equivalent cocycle.
\end{remark}
\subsection{Anderson Localization for Analytic Multi-Frequency Quantum Walks}
According to the analysis in Section \ref{secSCLE}, the Szeg\H{o} cocycle and the $n$-step transfer matrix of the matrix $\hat{\mathcal{E}}$ are
\begin{equation*}
\hat{M}(\omega,z;x)=\frac{1}{\rho(x)}
\left(
\begin{array}{cc}
\sqrt{z}&-\frac{\overline{\hat{\alpha}}(x)}{\sqrt{z}}\\
-\hat{\alpha}(x)\sqrt{z}&\frac{1}{\sqrt{z}}
\end{array}
\right)
\end{equation*}
and $\hat{M}_{n}(\omega,z;x)=\mathop{\prod}\limits_{j=n-1}^{0}\hat{M}(\omega,z;x+j\omega)$, respectively. In addition, the Lyapunov exponent of the matrix $\hat{\mathcal{E}}$ is
$$\hat{L}(w,z):=\lim_{n\rightarrow\infty}\frac{1}{n}\int_{\mathbb{T}^{d}}\log\|\hat{M}_{n}(\omega,z;x)\|dx.$$

By a simple calculation, we know that
\begin{equation*}
G(\omega,z;x)=
\left(
\begin{array}{cc}
\frac{1}{\sqrt{z}}&0\\
0&\sqrt{z}
\end{array}
\right)\hat{M}(\omega,z;x).
\end{equation*}
It is obvious that the difference between the Lyapunov exponent of the quantum walk $U_{\omega}(x)$ and the extended CMV matrix $\hat{\mathcal{E}}$ depends on the norm of the matrix
\begin{equation*}
\left(
\begin{array}{cc}
\frac{1}{\sqrt{z}}&0\\
0&\sqrt{z}
\end{array}
\right)=:M_{0}.
\end{equation*}
Let $C_{\mathrm{nor}}:=\|M_{0}\|$. Then one can obtain that
\begin{align*}
L^{U}(w,z)&=\lim_{n\rightarrow\infty}\frac{1}{n}\int_{\mathbb{T}^{d}}\log\|G_{n}(\omega,z;x)\|dx\\
&\leq\lim_{n\rightarrow\infty}\frac{1}{n}\int_{\mathbb{T}^{d}}\log((C_{\mathrm{nor}})^{n}\cdot\|\hat{M}_{n}(\omega,z;x)\|)dx\\
&=\log C_{\mathrm{nor}}+\hat{L}(\omega,z).
\end{align*}

Since $|z|=1$, then $\|M_{0}\|=1$ and  $\log C_{\mathrm{nor}}=0$. Therefore, we have that
\begin{equation}\label{eq:Lyap}
L^{U}(\omega,z)\leq \hat{L}(\omega,z).
\end{equation}

To this end, building upon Theorem \ref{main-theorem} and the unitary equivalence established in Section \ref{QWCMV}, we
now can present the main localization result for quantum walks with multi-frequency quasi-periodic coins.

\begin{theorem}[Anderson Localization for Quantum Walks]\label{thm:qw-al}
Consider a quantum walk $U_\omega(x)$ on $\ell^2(\mathbb{Z})\otimes\mathbb{C}^2$ with coins $C_n(x) = C(x+n\omega), \omega \in \mathbb{T}^d(p,q)$ satisfying:
\begin{enumerate}[(i).]
    \item $C: \mathbb{T}^d \to U(2)$ is real-analytic with analytic extension to $\mathbb{T}_h^d$;
    \item $\det C(x) = 1$ for all $x \in \mathbb{T}^d$;
     \item The Lyapunov exponent $L^U(\omega,z) > \gamma > 0$ for all $\omega \in \mathbb{T}^d(p,q)$ and all $z \in \mathcal{I}_C \subset \partial\mathbb{D}$, $\mathcal{I}_C$ is a compact interval.
\end{enumerate}
For fixed $x_0 \in \mathbb{T}^d$ and almost every $\omega \in \mathbb{T}^d(p,q)$, the quantum walk exhibits AL. Specifically, there exists a complete set of eigenfunctions $\{\psi_n\}$ such that
\begin{equation*}
U_\omega(x_0)\psi_n = e^{i\theta_n}\psi_n
\end{equation*}
with $\theta_n \in [0,2\pi)$ being the eigenphases, and the eigenfunctions decay exponentially:
\[
|\psi_n(k)| \leq c_n e^{-\gamma|k|/4} \quad \mathrm{for}\,\, \mathrm{all}\,\, k \in \mathbb{Z},
\]
where $c_n$'s are some constants.

Moreover, the unitary operator $U_\omega(x_0)$ has pure point spectrum in $\mathcal{I}_C \subset \partial\mathbb{D}$.
\end{theorem}

\begin{proof}
The proof proceeds via the unitary equivalence established in Section \ref{QWCMV}.

From  \eqref{eq:unitary-equiv}, we see that $U_\omega(x)$ is unitarily equivalent to the extended CMV matrix $\hat{\mathcal{E}}_\omega(x)$ via $\mathcal{D}$.

The Verblunsky coefficients $\hat{\alpha}_n(x)$ of $\hat{\mathcal{E}}_\omega(x)$  satisfy:
    \[
    \hat{\alpha}_{2n} = 0, \quad \hat{\alpha}_{2n+1} = -\frac{\lambda_{2n+1}}{\lambda_{2n+2}} \cdot c_n^{12}(x)
    \]
with $|\lambda_j| = 1$, hence $\hat{\alpha}_n(x)$ are real-analytic on $\mathbb{T}^d$.

The Lyapunov exponent relation \eqref{eq:Lyap} gives that
    \[
    \hat{L}(\omega,z) \geq L^U(\omega,z)> \gamma>0.
    \]

Then applying Theorem \ref{main-theorem} to $\hat{\mathcal{E}}_\omega(x)$ we can conclude that: for a.e.  $\omega \in \mathbb{T}^d(p,q)$, $\hat{\mathcal{E}}_\omega(x_0)$ has AL in $\mathcal{I}_C$.

By unitary equivalence, the eigenfunctions $\phi_n$ of  $\hat{\mathcal{E}}_\omega(x_0)$ correspond to eigenfunctions $\psi_n = \mathcal{D}\phi_n$ of $U_\omega(x_0)$ with identical decay rates.

The pure point spectrum follows from the completeness of eigenfunctions and unitary equivalence.
\end{proof}

\vskip1cm

\noindent{$\mathbf{Acknowledgments}$}

This work was supported in part by the  NSFC (No. 11571327, 11971059).

We extend our gratitude to Prof. Christopher Cedzich (Heinrich Heine Universit\"{a}t D\"{u}sseldorf) for his insightful suggestions on the comprehension of quantum walks, which greatly enhanced the rigorousness of our results.

\end{document}